\documentclass[10pt]{amsart}
\usepackage{amsthm,amsmath,amssymb,amsfonts}
\usepackage[american]{babel}
\usepackage{cite}
\usepackage{url}

\newtheorem*{theorem*}{Theorem}
\newtheorem{lemma}{Lemma}

\newtheorem{claim}{Claim}

\theoremstyle{definition}

\theoremstyle{remark}

\newtheorem*{problem*}{Problem}
\newtheorem*{problemdash*}{Problem$^{\prime}$}

\begin{document}


\newcommand{\NU}[1]{\text{\rm V}(#1)}
\newcommand{\cen}[2]{\text{\rm C}_{#1}(#2)}
\newcommand{\zen}[1]{\text{\rm Z}(#1)}
\newcommand{\opns}[1]{\text{\rm O}_{p^{\prime}}(#1)}
\newcommand{\lara}[1]{\langle{#1}\rangle}
\newcommand{\paug}[2]{\varepsilon_{#1}(#2)}
\newcommand{\ZZ}{\mathbb{Z}}
\newcommand{\ZZt}{\tilde{\mathbb{Z}}}
\newcommand{\QQ}{\mathbb{Q}}
\newcommand{\CC}{\mathbb{C}}
\newcommand{\irr}[1]{\text{\rm irr}(#1)}
\newcommand{\abs}[1]{\lvert#1\rvert}
\newcommand{\supp}[1]{\text{\rm supp}(#1)}


\title[Torsion units in integral group rings]
{On torsion units in integral group rings of Frobenius groups}

\author{Martin Hertweck}
\address{Universit\"at Stuttgart, Fachbereich Mathematik,
IGT, 
Pfaffenwald\-ring 57, 70569 Stuttgart, Germany}
\email{hertweck@mathematik.uni-stuttgart.de}



\subjclass[2000]{Primary 16S34, 16U60; Secondary 20C05}


\date{\today}

\keywords{integral group ring, torsion unit, Zassenhaus conjecture}

\begin{abstract}
For a finite group $G$, let $\ZZt$ be the semilocalization of $\ZZ$
at the prime divisors of $|G|$. If $G$ is a Frobenius group with
Frobenius kernel $K$, it is shown that each torsion unit in the 
group ring $\ZZt G$ which maps to the identity under the natural 
ring homomorphism $\ZZt G\rightarrow\ZZt G/K$ is conjugate
to an element of $G$ by a unit in $\ZZt G$.
\end{abstract}

\maketitle

\section{Introduction}\label{sec:intro}

A conjecture of H.~Zassenhaus from the 1970s asserts that
for a finite group $G$, every torsion unit in its integral group ring
$\ZZ G$ is conjugate to an element of $\pm G$ by a unit in the 
rational group ring $\QQ G$.
For known results on this conjecture the reader is referred to
\cite[Chapter~5]{Seh:93}, \cite[\S\,8]{Seh:03} and the more recent
work \cite{He:05,He:06}, \cite{CaMadelRio:12}.

The outstanding result in the field is Weiss's proof \cite{Wei:88,Wei:91}
that for a nilpotent group $G$, the conjecture is true. 
Weiss even proved (for nilpotent $G$) that a finite subgroup $H$ of 
$\NU{\ZZ G}$, the group of augmentation $1$ units of $\ZZ G$, 
is conjugate to a subgroup of $G$ by a unit in $\QQ G$, not without 
demonstrating that this, in his own words,  
``is still a rather crude description of the actual situation.''
We just mention that his results show that when the group $G$ is of
prime power order, or the subgroup $H$ is cyclic, conjugacy of $H$ to
a subgroup of $G$ already takes place in the units of $\ZZ_{p}G$, for 
each rational prime $p$, where $\ZZ_{p}$ denotes the $p$-adic 
integers \cite[p.~184]{Wei:91}.

We prove the following result.
Always, $G$ denotes a finite group, and we 
then write $\ZZt$ for the semilocalization of $\ZZ$ at the prime divisors
of $|G|$. (Explicitly, $\ZZt$ consists of all rational numbers with 
denominator not divisible by any prime dividing $|G|$.)

\begin{theorem*}
Let $G$ be a Frobenius group, with Frobenius kernel $K$. Then any
torsion unit in $\ZZt G$ which maps to the identity under the natural 
ring homomorphism $\ZZt G\rightarrow\ZZt G/K$ is conjugate to an 
element of $G$ by a unit in $\ZZt G$.
\end{theorem*}

That Frobenius kernels are nilpotent was proved by
J.~G.~Thompson (in his thesis).

Sehgal, apparently led by Weiss's work, raised the following question
\cite[Problem~35]{Seh:93}. 
\begin{problem*}
Let $N$ be a normal nilpotent subgroup of $G$. Then, is it true that
any torsion unit in $\ZZ G$ which maps to the identity under the natural
ring homomorphism $\ZZ G\rightarrow\ZZ G/N$ is conjugate to an element
of $G$ by a unit in $\QQ G$?
\end{problem*}

So we give an affirmative answer to this problem for Frobenius groups $G$.
In doing so, it will be important to take care of where conjugacy takes
place, and our result is really about conjugacy in all the $p$-adic 
group rings of $G$ (cf.\ \cite{Wei:91}, and the general remarks in 
\S\,\ref{sec:red}).
We first deal with the following problem, posed by Sehgal
at the 2002 conference ``Around group rings'' in Jasper, Alberta
\cite{AGR:02}.

\begin{problemdash*}
Let $N$ be a normal $p$-subgroup of $G$. Then, is it true that any 
torsion unit in $\ZZ G$ which maps to the identity under the natural
ring homomorphism $\ZZ G\rightarrow\ZZ G/N$ is conjugate to an element
of $G$ by a unit in $\ZZ_{p}G$?
\end{problemdash*}

A modest contribution to that was given in \cite{He:05,He:06}.
Here, we follow the approach given there, which is briefly  
described in \S\,\ref{sec:indec}. The proof of the theorem is then 
given in \S\,\ref{sec:proof}.

Apart from conjugacy questions, it is reasonable to look for isomorphisms.
In this context, the following is well known. Let $N$ be a normal 
nilpotent subgroup of $G$ and $U$ a finite group of units in $\ZZt G$ 
mapping to the identity under the map $\ZZt G\rightarrow\ZZt G/N$. 
Then $U$ is nilpotent. If $p$ is a prime dividing the order of $U$,
and $U_{p}$ and $N_{p}$ denote the Sylow $p$-subgroups of $U$ and $N$,
respectively, then $U_{p}$ is the kernel of the natural homomorphism
$U\rightarrow\NU{\ZZt G/N_{p}}$. (See \cite[Lemma~2.6]{ClWe:00} for details.)
When $N$ is a $p$-group, then it follows from \cite{Wei:88} that $U$
is conjugate to a subgroup of $N$ by a unit in $\QQ G$ 
(see \cite[Proposition~4.2]{He:06}). Thus in the general case, $U$ 
is isomorphic to a subgroup of $N$. Whether $U$ is conjugate
to a subgroup of $N$ by a unit in $\QQ G$ is Problem~36 in \cite{Seh:93}.
We were not able to answer this question in the present case
(that of Frobenius groups $G$).

We remark that torsion units in $\ZZt G$ share various properties 
with those of $\ZZ G$. We just mention that the partial augmentations
of torsion units in $\ZZt G$ are rational integers (this follows, for
example, immediately from \cite[Lemma~1]{Wei:91} since character values
are algebraic numbers).

Already known results on torsion units in the integral group ring of 
a Frobenius group $G$ include the following. Let $K$ be the Frobenius
kernel and $H$ a Frobenius complement of the Frobenius group $G$. 
Let $u$ be a torsion unit
in $\NU{\ZZ G}$. Then the order of $u$ divides either $|K|$ or $|H|$
(see \cite[Theorem~2.1]{JuPM:00}; for a related result, see 
\cite[Proposition~4.3]{Ki:06}). Hence $u$ either maps to the identity
under the natural map $\ZZ G\rightarrow\ZZ G/K$, or is conjugate to a
unit in $\ZZ H$ by a unit in $\QQ G$ (see \cite[Lemma~2.5]{DoJu:96} or 
\cite[(37.13)]{Seh:93}). The structure of Frobenius complements
is well known, so one might hope that for them, the Zassenhaus 
conjecture can be verified ``by inspection.'' 
For example, when $|H|$ is odd, then $H$ is metacyclic of the type 
treated in \cite{Val:94}, and any finite subgroup of
$\NU{\ZZ H}$ is conjugate to a subgroup of $H$ by a unit in $\QQ H$.
In general, a torsion unit in $\NU{\ZZ H}$ of prime power order
is conjugate to an element of $H$ by a unit in $\QQ H$.
This has been proved for $H$ solvable in \cite[Theorem~5.1]{DoJu:96},
and for ``most'' non-solvable $H$ in \cite{DoJuPM:97}, with the 
remaining case finally dealt with in \cite{BoHe:08}.

\section{Generalities about semilocal and $p$-adic conjugacy}\label{sec:red}

For convenience of the reader, we list a few basic facts about
conjugacy of finite subgroups in the group of units of a semilocal group
ring which are relevant to this paper.
Let $G$ be a finite group and $S$ a commutative ring.
There is a very general procedure for converting questions about 
conjugacy of finite subgroups of $\NU{SG}$ to questions about bimodules
(cf.\ \cite[(1.2)]{RoSc:87}). To illustrate, let $\alpha$ be a
group homomorphism from a finite group $H$ into $\NU{SG}$, and form
the right $S(G\times H)$-module $_{1}(SG)_{\alpha}$, which is $SG$ 
with group action $m(g,h)=g^{-1}m(h\alpha)$ for $m\in SG$, 
$g\in G$ and $h\in H$.
(All modules are understood to be $S$-linear.) Of course, this 
module can be viewed as an $SG^{\text{op}}$-$SH$ bimodule. 
Two such homomorphisms 
$\alpha,\alpha^{\prime}\colon H\rightarrow\NU{SG}$ are called 
{\em $S$-equivalent} (in \cite{Wei:91}) if a unit $u$ in $SG$ exists 
which satisfies 
$h\alpha=(h\alpha^{\prime})^{u}=u^{-1}(h\alpha^{\prime})u$ 
for all $h\in H$. Then it is easily seen that 
\[ \text{$\alpha$ and $\alpha^{\prime}$ are $S$-equivalent if and only 
if $_{1}(SG)_{\alpha}\cong{}_{1}(SG)_{\alpha^{\prime}}$.} \]
It is essential to spell out the correspondence explicitly. Suppose 
$\varphi\colon{}_{1}(SG)_{\alpha}\rightarrow{}_{1}(SG)_{\alpha^{\prime}}$
is an isomorphism; then $\varphi$ is given by right multiplication
with the unit $u=\varphi(1)$ of $SG$, and $u$ is a conjugating unit
as required for $S$-equivalence of $\alpha$ and $\alpha^{\prime}$.

The case of our interest is when $S=\ZZt$, the semilocalization of
$\ZZ$ at the prime divisors of $|G|$. 
Then conjugacy questions can be studied using general techniques 
from the local and global theory of integral representations
without going into complicated questions about locally free 
class groups (a general reference is \cite[Chapter~4]{CuRe:81}). 
Assume that $\alpha$ and 
$\alpha^{\prime}$ are homomorphisms from $H$ into $\NU{\ZZt G}$.
Note that for a coefficient ring $R$ containing $\ZZt$,
these homomorphisms can be viewed as homomorphisms into 
$\NU{RG}$. First, we have
$_{1}(\ZZt G)_{\alpha}\cong{}_{1}(\ZZt G)_{\alpha^{\prime}}$
if and only if 
$_{1}(\ZZ_{(p)}G)_{\alpha}\cong{}_{1}(\ZZ_{(p)}G)_{\alpha^{\prime}}$
for all prime divisors $p$ of $|H|$, where $\ZZ_{(p)}$ is the
localization of $\ZZ$ at $p$.
This depends on \cite[(31.15)]{CuRe:81} and 
Swan's theorem \cite[(32.1)]{CuRe:81}, as explained in 
\cite[Lemma~2.9]{He:05} (and the corollaries following it).
Second, we have 
$_{1}(\ZZ_{(p)}G)_{\alpha}\cong{}_{1}(\ZZ_{(p)}G)_{\alpha^{\prime}}$
if and only if 
$_{1}(\ZZ_{p}G)_{\alpha}\cong{}_{1}(\ZZ_{p}G)_{\alpha^{\prime}}$
(see \cite[(30.17)]{CuRe:81}). And the latter happens if
$_{1}(RG)_{\alpha}\cong{}_{1}(RG)_{\alpha^{\prime}}$ for a 
$p$-adic ring $R$, which, by definition, is the integral closure of 
$\ZZ_{p}$ in a finite extension field of the $p$-adic field $\QQ_{p}$
(see \cite[(30.25)]{CuRe:81}).

We shall use the following notation.
A finite cyclic group $C$ with (abstract) generator $c$ is fixed.
This group will keep in the background. All torsion units to be
considered are assumed to have order dividing the order of $C$.
When $u$ is a torsion unit in $\NU{SG}$, then 
$_{G}(SG)_{u}$ will denote the $S(G\times C)$-module
$_{1}(SG)_{\alpha}$, where $\alpha$ is the homomorphism
$C\rightarrow\NU{SG}$ which maps $c$ to $u$. 
In this way, we can compare the bimodules associated to various 
torsion units (of the same order) while avoiding to introduce
the relevant homomorphisms explicitly.
On occasion, for a subgroup $N$ of $G$, we also write $_{N}(SG)_{u}$
for $_{G}(SG)_{u}$ viewed as an $S(N\times C)$-module by restriction.
Also, for a direct summand $M$ of $_{G}(SG)_{u}$ (which, as an 
$S$-module, possibly may be considered as a summand of another 
bimodule), we write $_{G}M_{u}$ to emphasize the action to be 
considered.

\section{Indecomposable summands of certain bimodules}\label{sec:indec}

We recall some ideas from \cite[\S\S\,4,~5]{He:06} and \cite[\S\,4]{He:05}.
Let $N$ be a normal $p$-subgroup of the finite group $G$, and suppose 
that a non-trivial
torsion unit $u$ in $\NU{\ZZt G}$ is given which maps to the 
identity under the natural map $\ZZt G\rightarrow \ZZt G/N$.
Then $u$ is of order a power of $p$. For a Sylow $p$-subgroup $P$ of $G$, 
the bimodule $_{P}(\ZZ_{p}G)_{u}$ is a permutation lattice,
as a consequence of a theorem of Weiss \cite{Wei:88}.
It follows that $u$ is conjugate to an element $x$ of $N$
by a unit in $\QQ G$ (see \cite[Proposition~4.2]{He:06}). Moreover,
$_{N}(\ZZ_{p} G)_{u}\cong{}_{N}(\ZZ_{p} G)_{x}$ 
(see \cite[Claim~5.1]{He:06}), and so (see \cite[Claim~5.2]{He:06},
\cite[Corollary~3.2]{He:05}):
\begin{equation}\label{eq1}
_{G}(\ZZ_{p}G)_{u} \text{ is a direct summand of a
direct sum of copies of } _{G}(\ZZ_{p}G)_{x}.
\end{equation}

It will be convenient to work more generally with a $p$-adic ring $R$ 
as coefficient ring. 
It will be of importance that for a primitive idempotent $e$ in 
$R\cen{G}{x}$, the summand $RGe$ of 
$_{G}(RG)_{x}$ is indecomposable since its endomorphism ring 
$e(RG)^{\lara{x}}e$ is local (cf.\ \cite[\S\,4]{He:05}).
(Here, we write as usual $(RG)^{\lara{x}}$ for the fixed points in $RG$ 
under the conjugation operation of $\lara{x}$.)

We also recall the following simple fact (see \cite[Claim~4.3]{He:05}).
\begin{lemma}\label{elfact3}
Let $L$ be a normal subgroup of $G$ and let $e$ and $f$ be primitive 
idempotents in $RL$. Suppose that the $RG$-modules $RGe$ and $RGf$ 
have a common (non-zero) direct summand. Then $RGe\cong RGf$.
More precisely, $f=e^{gv}$ for some $g\in G$ and $v\in (RL)^{\times}$.
\end{lemma}

We obtain the following lemma (cf.\ \cite[Lemma~4.6]{He:05}).
\begin{lemma}\label{lem1}
Keep previously introduced notation.
Suppose that for a subgroup $L$ of $\cen{G}{x}$ which is normal in $G$,
there exist idempotents $e_{1},\dotsc,e_{s}$ in $RL$ such that 
$1=e_{1}+\ldots+e_{s}$ is an orthogonal decomposition into
primitive idempotents of $R\cen{G}{x}$. Then there exist
$g_{1},\dotsc,g_{s}\in G$ such that 
$_{G}(RG)_{u}\cong RGe_{1}^{g_{1}}\oplus\ldots\oplus RGe_{s}^{g_{s}}$ 
where each summand $RGe_{i}^{g_{i}}$ is considered as a submodule of
$_{G}(RG)_{x}$. In other words, $u$ is conjugate to 
$\sum_{j}e_{j}x^{g_{j}^{-1}}$ by a unit of $RG$.
\end{lemma}
\begin{proof}
We have $_{G}(RG)_{x}=RGe_{1}\oplus\ldots\oplus RGe_{s}$ where each
summand is an indecomposable $R(G\times C)$-module.
By \eqref{eq1}, and the Krull--Schmidt theorem, 
$_{G}(RG)_{u}\cong RGf_{1}\oplus\ldots\oplus RGf_{t}$ where the $f_{i}$
are taken from the set $\{e_{1},\dotsc,e_{s}\}$ (some of them may be 
equal), and each $RGf_{i}$ is considered as a submodule of $_{G}(RG)_{x}$.
Thus as $RG$-modules,
\begin{equation}\label{equat2}
RGe_{1}\oplus\ldots\oplus RGe_{s}=RG\cong RGf_{1}\oplus\ldots\oplus RGf_{t}.
\end{equation}
Renumbering if necessary, we can assume that $RGe_{1}$ and $RGf_{1}$ as
$RG$-modules have a direct summand in common. Then 
$f_{1}^{}=e_{1}^{g_{1}v}$ for some $g_{1}\in G$ and $v\in (RL)^{\times}$,
by Lemma~\ref{elfact3}. Obviously, $RGe_{1}^{g_{1}v}\cong RGe_{1}^{g_{1}}$
as submodules of $_{G}(RG)_{x}$.
Cancelling the isomorphic summands $RGe_{1}$ and $RGf_{1}$ in 
(\ref{equat2}) and continuing this way yields the isomorphism
$_{G}(RG)_{u}\cong\bigoplus_{i}RGe_{i}^{g_{i}}$ stated in the lemma. 
Finally, note that $\sum_{j}e_{j}x^{g_{j}^{-1}}$ is a torsion unit of $RG$
of order dividing the order of $x$ (since each $e_{j}$ commutes with
$x^{g_{j}^{-1}}$), and that there is an isomorphism
\newlength{\higgs}
\settoheight{\higgs}{${\sum_{j}e_{j}x^{g_{j}^{-1}}}$}
\[ _{G}(RG)_{u}\cong\bigoplus_{i=1}^{s}{}_{G}(RGe_{i}^{g_{i}})_{x}\cong 
\bigoplus_{i=1}^{s}{}_{\rule{0pt}{\higgs}G}(RGe_{i})_{\sum_{j}e_{j}x^{g_{j}^{-1}}} 
={}_{\rule{0pt}{\higgs}G}(RG)_{\sum_{j}e_{j}x^{g_{j}^{-1}}} \]
since an isomorphism between the summand 
$RGe_{i}^{g_{i}}$ of $_{G}(RG)_{x}$ and the summand $RGe_{i}$ of 
$_{\rule{0pt}{\higgs}G}(RG)_{\sum_{j}e_{j}x^{g_{j}^{-1}}}$ 
is given by right multiplication with $g_{i}^{-1}$. So $u$ is conjugate 
to $\sum_{j}e_{j}x^{g_{j}^{-1}}$ by a unit of $RG$.
\end{proof}

Suppose that $\cen{G}{x}$ has a normal $p$-complement which is normal
in $G$. So for $L=\opns{\cen{G}{x}}$, the quotient
$\cen{G}{x}/L$ is a $p$-group, and $L$ is a normal subgroup of $G$.
Note that this holds when $G$ is a Frobenius group
(then $L$ is the Sylow $p$-complement of the Frobenius kernel).
Then, if $R$ is suitably chosen, a primitive idempotent in $RL$ remains 
primitive in $R\cen{G}{x}$,
by Green's indecomposability theorem (see \cite[(19.23)]{CuRe:81}), 
and the lemma yields a conjugate of $u$ within the units of $RG$
of a rather specific shape (which is likely to lie within $R\cen{G}{x}$). 

\section{Proof of the theorem}\label{sec:proof}

We fix some notation for the rest of the paper.
Let $G$ be a Frobenius group with Frobenius kernel $K$.
Let $p$ be a prime divisor of $|K|$.
We write $K=N\times L$ with $N$ being the Sylow $p$-subgroup of $G$
and $L$ the Sylow $p$-complement of $N$ in $K$. 
Let $R$ be a $p$-adic ring containing a primitive $|L|$th root of unity.
Then the quotient field $E$ of $R$ is a splitting field for $L$.
Since $p$ does not divide $|L|$, the group ring $RL$ is isomorphic
to a direct sum of full matrix rings over $R$
(see \cite[V.5.7, V.5.12, V.12.5]{Hup:67}).
Let $\irr{L}$ be the set
of all irreducible ordinary characters of $L$ (with values in $E$).
Let $1=e_{1}+\ldots+e_{s}$ be an orthogonal decomposition of $1$ into
primitive idempotents of $RL$.

Even though we will make no use of it, 
we mention that for each idempotent $e_{i}$ 
distinct from the idempotent belonging to the principal character of $L$,
the $RG$-lattice $RGe_{i}$ is indecomposable (projective). For $e_{i}$ 
is primitive in $EL$, and when $H$ denotes a Frobenius complement 
in $G$, then $E(HL)e_{i}$ is a simple $E(HL)$-module (see, e.g., 
\cite[(6.34)]{Isa:94}). So $R(HL)e_{i}$ is an indecomposable 
$R(HL)$-lattice. 
This means that the image of $e_{i}$ under the natural map 
$RG\rightarrow RG/N=R(HL)$ is a primitive idempotent. But the 
kernel of this map is contained in the radical of $RG$, and so $e_{i}$
is primitive in $RG$ (see \cite[(5.26), (6.8)]{CuRe:81}).

For $m\in RG$ and $g\in G$, we write $\paug{G,g}{m}$ for the partial 
augmentation of $m$ with respect to the conjugacy class of $g$
in $G$. That is, when $m$ is written as an $R$-linear combination of
elements of $G$, then $\paug{G,g}{m}$ is the sum of the coefficients of 
the elements of the conjugacy class of $g$ in $G$.
Notation is chosen to emphasize that the partial augmentation 
is taken with respect to a conjugacy class of $G$,
since we shall also consider partial augmentations $\paug{L,g}{m}$
when $m\in RL$ and $g\in L$.
Note that $\varepsilon_{G,g}$ is a so-called trace function, that is, 
$\varepsilon_{G,g}$ is $R$-linear, and 
$\paug{G,g}{m_{1}m_{2}}=\paug{G,g}{m_{2}m_{1}}$ for all
$m_{1},m_{2}\in RG$.
For $m\in RG$, we shall write $\supp{m}$ for the support of $m$, which 
is the set of elements of $G$ whose coefficients in $m$ (again written
as above) are non-zero.

The following simple fact depends on $G$ being a Frobenius group.
For $1\neq x\in N$ and $a,a^{\prime}\in L$, the elements 
$xa$ and $xa^{\prime}$ are conjugate in $G$ if and only if 
the elements $a$ and $a^{\prime}$ are conjugate in $L$, and so
\begin{equation}\label{eq2}
\text{$\paug{G,xa}{xm}=\paug{L,a}{m}$ for all $1\neq x\in N$, $a\in L$
and $m\in RL$.}
\end{equation}
(For if $a^{\prime}\in\supp{m}$, then $(xa)^{g}=xa^{\prime}$ for some 
$g\in G$ if and only if $a^{h}=a^{\prime}$ for some $h\in L$, and the
coefficient of $xa^{\prime}$ in $xm$ is the same as the
coefficient of $a^{\prime}$ in $m$.)

This leads us to our first intermediate result. Note that
for a primitive idempotent $e$ of $RL$, we have $\lambda(e)=0$ 
for all $\lambda$ in $\irr{L}$ except one, when $\lambda(e)=1$.
We already noted in \S\,\ref{sec:indec} that
$_{G}(RG)_{x}=\bigoplus_{i=1}^{s}RGe_{i}$ is a direct sum
decomposition into indecomposable lattices.
Also note that if $e$ and $f$ are primitive idempotents of $RL$ and
$\lambda(e)=\lambda(f)$ for all $\lambda\in\irr{L}$,
then $e$ and $f$ are conjugate by unit in $RL$, and so
$RGe$ and $RGf$ are isomorphic summands of $_{G}(RG)_{x}$.

\begin{claim}\label{cla1}
Any torsion unit in $\ZZt G$ which maps to the identity under the natural
map $\ZZt G\rightarrow\ZZt G/N$ is conjugate to an element of $N$ 
by a unit in $RG$.
\end{claim}
\begin{proof}
Let $u$ be such a torsion unit, $u\neq 1$, and let $x$ be an element in 
$N$ to which $u$ is conjugate by a unit in $\QQ G$ (such an element
exists, as recorded in \S\,\ref{sec:indec}).
By Lemma~\ref{lem1}, and the remark following its proof, there exist
$g_{1},\dotsc,g_{s}\in G$ such that $u$ is conjugate to 
$\sum_{j}e_{j}x^{g_{j}^{-1}}$ by a unit in $RG$.
Set $m=\sum_{j}e_{j}^{g_{j}}$. Let $a\in L$. We have
$\paug{G,xa}{u}=\sum_{j}\paug{G,xa}{e_{j}x^{g_{j}^{-1}}}=
\sum_{j}\paug{G,xa}{e_{j}^{g_{j}}x}=
\paug{G,xa}{xm}=\paug{L,a}{m}$ by \eqref{eq2}.
Since $u$ is conjugate to $x$ by a unit in $\QQ G$, we have
$\paug{G,x}{u}=1$ and $\paug{G,xa}{u}=0$ if $a\neq 1$.
Thus $m\in 1+{}_{R}\lara{gh-hg\mid g,h\in L}$ and so
$\sum_{j}\lambda(e_{j}^{g_{j}})=\lambda(m)=\lambda(1)$ 
for all $\lambda\in\irr{L}$. It follows that
\[ _{G}(RG)_{u}\cong RGe_{1}^{g_{1}}\oplus\ldots\oplus RGe_{s}^{g_{s}}
\cong RGe_{1}\oplus\ldots\oplus RGe_{s}= {}_{G}(RG)_{x} \]
where all the summands are considered as submodules of
$_{G}(RG)_{x}$ (the first isomorphism is stated in Lemma~\ref{lem1}). 
This is what we wanted to prove. 
\end{proof}

We proceed to prove the theorem. Let $w$ be a torsion unit of $\ZZt G$
which maps to the identity under the natural map 
$\ZZt G\rightarrow\ZZt G/K$. We wish to show that $w$ is conjugate to 
an element of $K$ by a unit in $\ZZt G$. We can assume that its 
$p$-part $u$ (say) is $\neq 1$. By the general facts recorded in
\S\,\ref{sec:red}, it then suffices to show
that $w$ is conjugate to an element of $K$ by a unit in $RG$.
In view of Claim~\ref{cla1}, we can also assume that its 
$p^{\prime}$-part $y$ (say) is $\neq 1$, and by induction on the number of 
prime divisors of the order of $w$, we can assume that $y\in L$.
Note that $w=uy=yu$. 

By Claim~\ref{cla1}, $u=x^{\mu}$ for some $x\in N$ and a unit $\mu$
in $RG$. We next observe that $\mu$ can be chosen such that $y^{\mu}$
is contained in $RL$, by choosing a suitable isomorphism 
$_{G}(RG)_{u}\cong {}_{G}(RG)_{x}$. As $y\in (RG)^{\lara{u}}$,
the orthogonal decomposition of $1$ into the sum of the
primitive idempotents of $R\lara{y}$ can be refined to 
an orthogonal decomposition $1=\nu_{1}+\ldots +\nu_{t}$ with
primitive idempotents $\nu_{j}$ of $(RG)^{\lara{u}}$. 
Then $_{G}(RG)_{u}=\bigoplus_{j=1}^{t}RG\nu_{j}$ is a direct sum
decomposition into indecomposable lattices. Furthermore, right 
multiplication with $y$ on $RG$ leaves the summands fixed, the action of
$y$ on each of them given by multiplication with a root of unity in $R$.
Since $_{G}(RG)_{u}$ and $_{G}(RG)_{x}$ are isomorphic, 
the Krull--Schmidt theorem implies that $s=t$, and that 
an isomorphism between the two lattices, which in any case is given by
multiplication with a unit $\mu$ in $RG$ such that $x=u^{\mu}$,
can be chosen such that each $RG\nu_{j}$ is mapped to some $RGe_{i}$. 
Then $y^{\mu}$ acts on each $RGe_{i}$ from right by multiplication 
with a root of unity $\zeta_{i}$. 
For if $RG\nu_{j}$ is mapped to $RGe_{i}$ under multiplication with
$\mu$, then $\nu_{j}^{\mu}=e_{i}$, and 
$e_{i}y^{\mu}=e_{i}\mu^{-1}y\mu=\mu^{-1}(\nu_{j}y)\mu=
\zeta_{i}\nu_{j}^{\mu}=\zeta_{i}e_{i}$
where $\zeta_{i}$ is the root of unity with $\nu_{j}y=\zeta_{i}\nu_{j}$.
Consequently, 
$y^{\mu}=\sum_{i}\zeta_{i}e_{i}\in RL$ and 
$w^{\mu}=(uy)^{\mu}=xy^{\mu}\in x(RL)$.
Note that this implies that if $\paug{g}{w}\neq 0$ for some $g\in G$,
then $g$ is conjugate to $xa_{0}$ for some $a_{0}\in L$. Furthermore,
by \eqref{eq2}, 
\begin{equation}\label{eq3}
\paug{G,xa}{w}=\paug{L,a}{y^{\mu}} \text{ for all } a\in L.
\end{equation}

\begin{claim}\label{cla2}
Suppose that $w$ is conjugate to an element $k$ of $K$ by a unit in
$EG$. Then $w$ is also conjugate to $k$ by a unit in $RG$.
\end{claim}
\begin{proof}
We can assume that $k$ lies in the support of $w^{\mu}$, and then 
$k=xa_{0}$ with $a_{0}\in L$. By assumption, $\paug{G,k}{w}$ is 
the only nonzero partial augmentation of $w$.
Suppose that $\paug{L,a}{y^{\mu}}\neq 0$ for some $a\in A$.
Then $\paug{G,xa}{w}\neq 0$ by \eqref{eq3}, so $xa$ is conjugate to
$xa_{0}$ in $G$, and hence $a$ is conjugate to $a_{0}$ in $L$
(as $G$ is a Frobenius group). Thus $\paug{L,a_{0}}{y^{\mu}}$
is the only nonzero partial augmentation of $y^{\mu}$ as an 
element of $RL$. Let us convince ourselves that the same applies to 
powers of $y^{\mu}$. 
Let $a\in L$ and $i$ a natural number not divisible by $p$.
By \eqref{eq3}, which applies analogously to $w^{i}$, and \eqref{eq2},
we have $\paug{L,a}{(y^{\mu})^{i}}=\paug{G,x^{i}a}{w^{i}}=
\paug{G,x^{i}a}{k^{i}}=\paug{G,x^{i}a}{x_{}^{i}a_{0}^{i}}=
\paug{L,a}{a_{0}^{i}}$.
So $\paug{L,a}{(y^{\mu})^{i}}$ is nonzero only if $a$ is conjugate to
$a_{0}^{i}$ in $L$. A standard argument from character theory now shows
that $y^{\mu}$ is conjugate to $a_{0}$ by a unit in $EL$
(see \cite[Lemma (37.6)]{Seh:93}). But then $y^{\mu}$, being of order 
not divisible by $p$, is also conjugate
to $a_{0}$ by a unit $\tau$ in $RL$ (see \cite[Lemma~2.9]{He:06})
and $w^{\mu\tau}=(xy^{\mu})^{\tau}=xa_{0}$.
\end{proof}

It remains to verify the assumption of the claim, that
$w$ is conjugate to an element of $K$ by a unit in $EG$. 
To this aim, we can inductively assume that
the theorem is true for Frobenius groups of order strictly smaller than
$|G|$. Then we can state the following.

\begin{claim}\label{cla3}
Suppose that $x\not\in\zen{N}$. Then $w$ is conjugate to an element of $K$ 
by a unit in $EG$.
\end{claim}
\begin{proof}
The proper quotient $\bar{G}=G/\zen{N}$ of $G$ is a Frobenius group.
(Suppose that $(\zen{N}n)^{g}=\zen{N}n$ for $n\in N$ and $g$ of prime
order $r$ in a Frobenius complement of $G$. Then $r\neq p$ and so 
$(zn)^{g}=zn$ for some $z\in\zen{N}$. It follows that $n=z^{-1}$
and $\zen{N}n=\zen{N}$.)
By the induction hypothesis, $\bar{w}$ is conjugate to an element
of $\bar{G}$ by a unit in $E\bar{G}$.
As in the proof of the previous claim, it follows that 
$\bar{y}^{\bar{\mu}}$ is conjugate, for some $a_{0}\in L$,
to $\bar{a}_{0}$ by a unit in $E\bar{L}$.
Since $\bar{L}\cong L$ it follows that 
$y^{\mu}$ is conjugate to $a_{0}$ by a unit in $EL$, and
consequently $w$ is conjugate to $xa_{0}$ by a unit in $EG$.
\end{proof}

Note that for any other prime divisor $q$ of the order of $w$ (distinct
from $p$), the claim also holds with $N$ replaced by the Sylow
$q$-subgroup of $K$, and with $x$ replaced by an element of this 
subgroup to which the $q$-part of $u$ is conjugate by a unit of $EG$. 
Thus we can assume, for the rest of the proof, that $y\in\zen{L}$.

We state a well known result (due to Berman and Higman) about 
vanishing of the $1$-coefficient for certain torsion units, but 
provide the proof anyway.

\begin{lemma}\label{BermanHigman}
Let $H$ be a finite group and $S$ an integral domain of characteristic
$0$. Let $v$ be a torsion unit in $SH$ of augmentation $1$ such that
$0\neq \paug{a}{v}\in\ZZ$ for some $a\in\zen{G}$. Then $v=a$.
\end{lemma}
\begin{proof}
Replacing $v$ by $va^{-1}$ (which can be done since $a$ is central)
we obtain $0\neq \paug{1}{v}\in\ZZ$ 
and have to show that $v=1$. Observe that $v$ is a unit in $FH$, where
$F$ is the finitely generated field obtained from the rational numbers
by adjoining all the non-zero coefficients of $v$. 
Since $F$ is embeddable into the field of complex numbers $\CC$, we may
assume that $v\in\CC H$. Let $\theta$ be the character of the regular 
representation of $H$ on $\CC H$. Then $\theta(v)=\abs{H}\paug{1}{v}$.
Since $v$ is a torsion unit (and hence a representing matrix can be
diagonalized), $\abs{\theta(v)}\leq\theta(1)=\abs{H}$ with equality
if and only if $v$ is a root of unity, and then $v=1$ as $v$ has 
augmentation $1$. That equality must hold follows from the assumption 
on $\paug{1}{v}$.
\end{proof}

Since $\paug{G,y}{y^{\mu}}=1$, there is $g\in G$ such that 
$y^{g}\in\supp{y^{\mu}}$. Then $xy^{g}\in\supp{xy^{\mu}}$. 
Since no distinct $G$-conjugate of $xy^{g}$ lies in the
support of $xy^{\mu}$ as $y^{g}\in\zen{L}$ and $y^{\mu}\in RL$
(remember that $G$ is a Frobenius group), it follows with 
\eqref{eq3} that
$\paug{L,y^{g}}{y^{\mu}}=\paug{G,xy^{g}}{w}=\paug{G,xy^{g}}{xy^{\mu}}\neq 0$.
Hence it follows from Lemma~\ref{BermanHigman} (applied with $H=L$,
$S=R$, $v=y^{\mu}$ and $a=y^{g}$) that $y^{\mu}=y^{g}$ and 
$w^{\mu}=xy^{\mu}=xy^{g}\in K$. The theorem is proved.




\providecommand{\bysame}{\leavevmode\hbox to3em{\hrulefill}\thinspace}
\providecommand{\MR}{\relax\ifhmode\unskip\space\fi MR }
\providecommand{\MRhref}[2]{%
  \href{http://www.ams.org/mathscinet-getitem?mr=#1}{#2}
}
\providecommand{\href}[2]{#2}

\end{document}